\pdfoutput=1
\RequirePackage{ifpdf}
\ifpdf 
\documentclass[pdftex]{sigma}
\else
\documentclass{sigma}
\fi

\usepackage{amscd}
\usepackage[all]{xy}
\usepackage{tikz}
\usepackage{url}

\newcommand{\RR}{\mathbb{R}}

\newcommand{\PP}{\mathbb{P}}
\newcommand{\AAA}{\mathbb{A}}

\def\cB{{\mathcal B}}
\def\cC{{\mathcal C}}

\def\cI{{\mathcal I}}

\def\cX{{\mathcal X}}

\def\dd{{\mathbf d}}
\def\nn{{\mathbf n}}
\def\rr{{\mathbf r}}
\def\DD{{\mathbf D}}
\def\NN{{\mathbf N}}
\def\RR{{\mathbf R}}
\def\kk{{\mathbf k}}

\numberwithin{equation}{section}
\newtheorem{Theorem}{Theorem}[section]
\newtheorem{Lemma}[Theorem]{Lemma}
{ \theoremstyle{definition}
\newtheorem{Example}[Theorem]{Example}}

\begin{document}
\allowdisplaybreaks

\newcommand{\arXivNumber}{1710.02863}

\renewcommand{\PaperNumber}{031}

\FirstPageHeading

\ShortArticleName{Cartan Prolongation of a Family of Curves Acquiring a Node}

\ArticleName{Cartan Prolongation of a Family\\ of Curves Acquiring a Node}

\Author{Susan Jane COLLEY~$^\dag$ and Gary KENNEDY~$^\ddag$}

\AuthorNameForHeading{S.J.~Colley and G.~Kennedy}

\Address{$^\dag$~Department of Mathematics, Oberlin College, Oberlin, Ohio 44074, USA}
\EmailD{\href{mailto:sjcolley@math.oberlin.edu}{sjcolley@math.oberlin.edu}}
\URLaddressD{\url{https://www.oberlin.edu/susan-colley}}

\Address{$^\ddag$~Ohio State University at Mansfield, 1760 University Drive, Mansfield, Ohio 44906, USA}
\EmailD{\href{mailto:kennedy@math.ohio-state.edu}{kennedy@math.ohio-state.edu}}
\URLaddressD{\url{https://u.osu.edu/kennedy.28/}}

\ArticleDates{Received October 27, 2017, in final form April 03, 2018; Published online April 07, 2018}

\Abstract{Using the monster/Semple tower construction, we study the structure of the Cartan prolongation of the family $x_1x_2 = t$ of plane curves with nodal central member.}

\Classification{58A30; 53A55; 58K50; 14D06; 14H99}

\section{Introduction} \label{intro}

The familiar process of implicit differentiation, when applied $k$ times to the equation $f(x,y)=0$ of a nonsingular plane curve $C$, yields a system of equations
\begin{gather}
f = 0, \nonumber\\
f_x + f_y y' = 0, \nonumber\\
f_{xx} + (2f_{yx} + f_{yy} y') y' + f_y y'' = 0, \nonumber\\
\quad \vdots\nonumber\\
f_{x \cdots x} + \cdots + f_{y} y^{(k)} = 0.\label{implicit}
\end{gather}
One can interpret (\ref{implicit}) as a system defining a curve in a $(k+2)$-dimensional manifold. Indeed, at each point of $C$ where $f_y \neq 0$, the equations of (\ref{implicit}) determine unique values of $y', y'', \ldots, y^{(k)}$, which we will call the \emph{curvilinear data}. At points where $f_y=0$ (i.e., at which the tangent line is vertical), one can reverse the roles of the variables, and the resulting system will determine values of $x'={d}x/{d}y, x'', \ldots, x^{(k)}$. In the literature of differential geometry, this process of obtaining a~curve in a higher-dimensional space is often called \emph{prolongation}.

The notion of prolongation dates back to Cartan's 1914 paper~\cite{MR1504722}. It was introduced as a means of understanding when two systems of differential equations should be considered to be equivalent. As explained by Bryant (in Section~6.2 of the unpublished lecture notes~\cite{BryantLectureNotes}), ``Intuitively, prolongation is just differentiating the equations you have and then adjoining those equations as new equations in the system''. Further discussion of the idea can be found in \cite[Section~4.2]{MR1240644}.

One can also prolong a singular curve. To make sense of this assertion, one needs to understand how to deal with the possibility that the higher-order curvilinear data may ``become infinite''. In other words, one needs to explain how to compactify the manifold of nonsingular curvilinear data. There is an elegant compactification, first explained in the literature of diffe\-ren\-tial geometry by Montgomery and Zhitomirskii~\cite{MR1841129,MR2599043}. Starting with an arbitrary surface $S$, their \emph{monster tower} is a sequence
\begin{gather*}
\cdots \to S(k) \to S(k-1) \to \cdots \to S(2) \to S(1) \to S(0)=S
\end{gather*}
in which each $S(k) \to S(k-1)$ is a fiber bundle with fiber $\PP^{1}$. The equations of (\ref{implicit}) can now be interpreted as a calculation within one of two \emph{regular charts} of $S(k)$. A similar calculation can be carried out for any curve in~$S$, and for a singular curve it will lead to curves that leave the regular charts. Their work on the monster construction was motivated in part by a body of literature about Goursat distributions, including classification of their possible singularities and the appearance of moduli in this classification; as representative contributions we mention~\cite{MR507769} and~\cite{MR2444398}.

At about the time of the publication of the Montgomery--Zhitomirskii monograph~\cite{MR2599043}, it was observed by Alex Castro that in fact the monster construction was already known in algebraic geometry, in literature beginning with Gherardelli~\cite{MR0017961} and Semple~\cite{MR0061406}. In this strand of literature, instead of ``prolongation'' one will find ``Nash blowup'' and ``lifting''; here the tower is called the \emph{Semple bundle tower}. Our own contributions were in the papers \cite{MR1100358,MR1143546,MR1287696}.

In \cite{MR2434453}, Cartan prolongation is applied to resolution of plane curves. This study naturally leads to the following problem, stated as Problem A in the last section of the cited paper: ``Investigate the extent to which deformation and prolongation of curves commute''. In this note we offer a response, based on the standard theory of extension of a flat family over a~punctured one-dimensional base, as in Proposition 9.8 of Hartshorne's text~\cite{MR0463157}.

For the moment, we have just written an account of the theory in the simplest nontrivial case: a family $\cX$ of curves $x_1 x_2=t$. Here $x_1$ and $x_2$ are local coordinates on a neighborhood $U$ of a point on a surface~$S$, and~$t$ is a local coordinate on a neighborhood $T$ of $0$ on the line. (We use subscripted notations to be consistent with the coordinates introduced in Section~\ref{charts}.) The central member of the family is the \emph{nodal curve} $x_1 x_2=0$. We are asking what happens when one tries to prolong all members of the family. For $t \neq 0$ the process of prolongation, carried out~$k$ times, gives us a~family of nonsingular curves on the monster space~$S(k)$. What happens over $t=0$?

Hartshorne's Proposition~9.8 says that there is a unique extension to a flat family defined even over $t=0$. The total space of the family
of $k$th order prolongations (including the central member) will be denoted by $\cX(k)$; it is a subvariety of $S(k) \times T$. The members of the family will be denoted by $X(k)_t$; we call the central member $X(k)_0$ the \emph{prolongation of the nodal curve}. Here we will work out an explicit description of the central member: it consists of $2^k+1$ irreducible components, which we call \emph{twigs}. The twigs meet at ordinary nodes, forming a chain; only the two curves at the ends will be reduced schemes. When reduced, each twig (except the two at the ends) is a copy of the projective line. The two end twigs are the Cartan prolongations of $x_2=0$ (which we call the \emph{left end}) and of $x_1=0$ (the \emph{right end}). The other twigs will be called \emph{interior twigs}. Fig.~\ref{thirdprolongation} shows the case $k=3$.

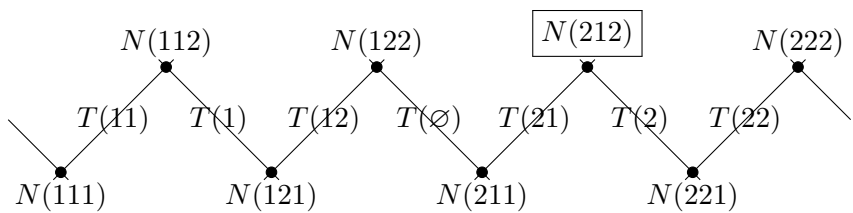
\begin{figure}[htbp]\centering
\begin{tikzpicture}
\draw (0.8,-0.8) --(1.6,-1.6);
\draw (2.8,0) --(4.4,-1.6);
\draw (5.6,0) --(7.2,-1.6);
\draw (8.4,0) --(10.0,-1.6);
\draw (11.2,0) --(12.0,-0.8);

\draw (1.4,-1.6) --(3.0,0);
\draw (4.2,-1.6) --(5.8,0);
\draw (7.0,-1.6) --(8.6,0);
\draw (9.8,-1.6) --(11.4,0);

\draw[fill] (1.5,-1.5) circle [radius=0.07];
\draw[fill] (4.3,-1.5) circle [radius=0.07];
\draw[fill] (7.1,-1.5) circle [radius=0.07];
\draw[fill] (9.9,-1.5) circle [radius=0.07];

\draw[fill] (2.9,-0.1) circle [radius=0.07];
\draw[fill] (5.7,-0.1) circle [radius=0.07];
\draw[fill] (8.5,-0.1) circle [radius=0.07];
\draw[fill] (11.3,-0.1) circle [radius=0.07];

\node [below] at (1.5,-1.5) {$N(111)$};
\node [below] at (4.3,-1.5) {$N(121)$};
\node [below] at (7.1,-1.5) {$N(211)$};
\node [below] at (9.9,-1.5) {$N(221)$};

\node [above] at (2.9,-0.1) {$N(112)$};
\node [above] at (5.7,-0.1) {$N(122)$};
\node [above] at (8.5,-0.1) {$\boxed{N(212)}$};
\node [above] at (11.3,-0.1) {$N(222)$};

\node at (2.2,-0.8) {$T(11)$};
\node at (3.6,-0.8) {$T(1)$};
\node at (5.0,-0.8) {$T(12)$};
\node at (6.4,-0.8) {$T(\varnothing)$};
\node at (7.8,-0.8) {$T(21)$};
\node at (9.2,-0.8) {$T(2)$};
\node at (10.6,-0.8) {$T(22)$};
\end{tikzpicture}

\caption{The third prolongation of the nodal curve, with nodes labeled lexicographically; the labeling of the twigs will be explained at the end of Section \ref{schemestruc}. The box is related to Fig.~\ref{binomstolength3} and to the running example, which begins with Example \ref{beginrun}.} \label{thirdprolongation}
\end{figure}

Again by standard theory, the flat family $\cX(k)$ of curves gives a family of 1-dimensional cycle classes. For $t\ne 0$ this class is just the fundamental class of the $k$th prolongation of $X_t$. The class associated to the central member will be denoted by $Z(k)_0$. As we will see,
it is a linear combination of fundamental classes of the $2^k+1$ curves mentioned above, and (except for the two curves at the end) the coefficients will be greater than~1.

We begin in Section \ref{monster} by briefly recalling the construction of the monster tower and explaining prolongation of a curve into it. In Section~\ref{charts} we describe a natural system of charts on the spaces in the monster tower, citing the general description of \cite{MR3720327} but specializing to the case of a 2-dimensional base. In Section \ref{schemestruc} we give equations defining $\cX(k)$ in each of these charts;
from these equations we immediately deduce the irreducible components, i.e., the twigs. Section~\ref{cyclestruc} describes the multiplicities of these twigs. Again referring to our earlier work (joint with Castro and Shanbrom) in~\cite{MR3720327}, in Section \ref{strat} we analyze how the twigs meet the strata of the natural coarse stratification of the monster space. We illustrate the ideas by a running example, beginning with Example~\ref{beginrun}. Finally in Section~\ref{speculate}, we briefly speculate about what one should expect if one carries out a similar analysis for a locally irreducible curve singularity such as $y^2=x^n$.

\section{The monster tower; prolongation of curves} \label{monster}
We begin by briefly explaining the construction of the monster tower, closely following the account in Section~1 of our recent paper with Castro and Shanbrom~\cite{MR3720327}. That paper also gives a~brief history of the construction, which was discovered independently in two separate
strands of the mathematical literature. (For other treatments, see \cite{MR1841129,MR2599043}, or our ancient papers~\cite{MR1100358,MR1143546,MR1287696}.)

Suppose that $M$ is a smooth manifold or nonsingular algebraic variety over an algebraically closed field~$\kk$ of characteristic~0. Suppose that
$\cB$ is a rank $b$ subbundle of its tangent bundle~$TM$. Let $\widetilde{M}=\PP\cB$, the total space of the projectivization of the bundle, and let $\pi \colon \widetilde{M} \to M$ be the projection. A point $\widetilde{p}$ of $\widetilde{M}=\PP\cB$ over $p \in M$ represents a line inside the fiber of $\cB$ at $p$, and since $\cB$ is a subbundle of $TM$, this is a \emph{tangent direction} to~$M$ at~$p$. Let
\begin{gather*}
{d}\pi \colon \ T{\widetilde{M}} \to \pi^*TM
\end{gather*}
denote the derivative map of $\pi$. A tangent vector to $\widetilde{M}$ at $\widetilde{p}$ is said to be a \emph{focal vector} if it is mapped by ${d}\pi$ to a tangent vector at~$p$ in the direction represented by $\widetilde{p}$; in particular a vector mapping to the zero vector (called a \emph{vertical vector}) is considered to be a focal vector. The set of all focal vectors forms a subbundle $\widetilde \cB$ of~$T{\widetilde{M}}$, called the \emph{focal bundle}; its rank is again~$b$. Thus we can iterate this construction to obtain a tower of spaces (i.e., smooth manifolds or nonsingular algebraic varieties) together with their associated bundles.

If we begin this construction with a surface $S$, taking $\cB$ to be its tangent bundle, then the resulting tower
\begin{gather*}
\cdots \to S(k) \xrightarrow{\pi_k} S(k-1) \xrightarrow{\pi_{k-1}} \cdots \to S(2) \to S(1) \to S(0)=S
\end{gather*}
is called the \emph{monster tower}, and the spaces $S(1), S(2), \ldots$ are called \emph{monster spaces}. Observe that each $S(k)$ is the total space of a~$\PP^{1}$-bundle over~$S(k-1)$; in particular, $S(1)$ is the total space of the tangent bundle~$\PP TS$. Sometimes we say that~$S(k)$ is the monster space at \emph{level~$k$}. If $p$ is a point of~$S(k)$, we will say that it \emph{lies over}~$\pi_k(p)$ and over $\pi_{k-1}\pi_k(p)$, etc. The associated rank~2 focal bundle on $S(k)$ will be denoted by~$\cB(k)$.

Now consider a nonsingular point $p$ of a curve $C$ in $S(k)$. We can associate to $p$ the point of~$\PP TS(k)$ representing the tangent direction of $C$ at~$p$. Thus, away from singularities, we have a~curve $\widetilde{C}$ in~$\PP T{S(k)}$, the \emph{prolongation} or \emph{lift}. Now suppose that $C$ itself is the prolongation of a curve from $S(k-1)$. Note that this curve is $\pi_k(C)$; thus the tangent vectors to $\widetilde{C}$ are focal vectors. Therefore $\widetilde{C}$ is in fact a curve in $S(k+1)$.

We also want to lift a singular curve $C$ in $S(k)$, and we do so by fiat: we lift at all nonsingular points of $C$, and then take the closure. If, away from the singularity, $C$ is obtained from lifting a curve in $S(k-1)$, then again $\widetilde{C}$ lies entirely within $S(k+1)$. For example, the lift of the plane curve $xy=0$ will consist of two irreducible curves, the lifts of its individual nonsingular components $x=0$ and $y=0$. Note that this is \emph{not} what we have previously called the prolongation of the nodal curve, since here we are considering the curve by itself, without reference to the family in which it is situated. Similarly, the lift of the cuspidal curve $y^2=x^3$ will be a nonsingular curve; note that it is tangent to the fiber of $S(1)$ over the origin.

Consider the $\PP^1$-fiber of $S(k)$ over a point of $S(k-1)$; following \cite[Definition~2.17]{MR2599043}, we call it a \emph{vertical curve}. Note that its prolongation is\ a curve in $S(k+1)$; indeed, all of the tangent vectors are vertical. This curve will be called the \emph{prolongation of a vertical curve}, and we apply the same terminology to the curves obtained by further prolongation.

In the literature of differential geometry, e.g., \cite{MR2444398}, the monster tower emerges from a study of Goursat distributions. The fundamental work \cite{MR1841129} shows that locally each Goursat distribution arises from the prolongation of a suitable curve on a surface. This construction is not canonical; that is, two different curves may lead to the same Goursat distribution. Perhaps the simplest example begins with the singular parametrized curve $x = t^2$, $y = t^3$, whose prolongation yields the curve
\begin{gather*}
x = t^2, \qquad y = t^3, \qquad y' = \frac{3}{2}t,
\end{gather*}
where $y' = {d}y/{d}x$. This is a curve on a chart of $S(1)$. On may also begin with the nonsingular curve $\overline{x} = \frac{3}{2}t$, $\overline{y} = t^3$, and prolong it to
\begin{gather*}
\overline{x} = \frac{3}{2}t, \qquad \overline{y} = t^3, \qquad \overline{y}' = 2 t^2,
\end{gather*}
where $\overline{y}' = {d}\overline{y}/{d}\overline{x}$. The declarations $\overline{x} = y'$, $\overline{y} = y$, $\overline{y}' = 2x$ identify the two prolongations. Thus in the differential-geometric literature one sees no distinction between the resulting Goursat distributions. For us, however, the fundamental object is the curve on the base rather than the distribution at level~$1$.

We were aware of this ambiguity when we wrote \cite{MR1143546}, in which we observed (on page~35) that the Semple bundle tower over the projective plane $\PP^2$ is the same as the Semple bundle tower over the dual projective plane $\check{\PP}^2$. This means that if one considers, as in classical projective geometry, a curve $C \subset \PP^2$ and its dual curve $\check{C} \subset \check{\PP}^2$, they have the same prolongation, considered as a~curve in the incidence correspondence of points and lines, which can be identified with $\PP^2(1)$. This is the key fact in understanding the classical Pl\"{u}cker formulas (see \cite[Section~9.1]{MR2975988}).

\section{Coordinate charts on the monster tower} \label{charts}

We now recall a natural system of $2^k$ coordinate charts on the $k$th monster space. They are essentially the coordinates first described in~\cite{MR507769} and~\cite{MR704040} in a study of Pfaffian systems; in the subsequent literature of differential geometry, they are known as Kumpera--Ruiz coordinates. Independently and in a seemingly different context, we developed coordinate systems for the Semple bundle tower over a base surface \cite{MR1100358,MR1143546,MR1287696}. The paper \cite{MR1841129} shows that the Kumpera--Ruiz coordinates are coordinates on the monster spaces; Castro's later realization that the Semple bundle and monster tower constructions are the same thus implies that the two notions of coordinate systems must agree. These coordinates were extended to towers over higher-dimensional base manifolds and systematized in~\cite{MR2262175}. In \cite{MR3720327}, we explained a convenient general method for naming these coordinates. We use that method here, specializing to the case where the base is a surface.

Each chart is a copy of $U \times \AAA^k$, the product of the base neighborhood $U$ and $k$-dimensional affine space, and on each chart there are $k+2$ coordinate functions: the pullback of $x_1$ and $x_2$ from $U$, together with $k$ affine coordinates. By a recursive procedure, two of these coordinates will be designated as \emph{active coordinates}. One will be identified as the \emph{new coordinate} and denoted by $\nn$, and the other will be identified as the \emph{retained coordinate} and denoted by~$\rr$. In addition (for $k>0$) a third coordinate will be designated as the \emph{deactivated coordinate} and denoted by~$\dd$. At each point~$p$ of the chart, the fiber of $\cB(k)$ consists of tangent vectors for which
either~$d\nn$ or~$d\rr$ is nonzero. Here is the recursive procedure: beginning with a chart with~$\nn$, $\rr$, and $\dd$ (plus $k-1$ unnamed coordinates), create two charts at the next level by choosing one of the following two options:
\begin{itemize}\itemsep=0pt
\item Assuming the differential $d\rr$ is nonzero, let $\NN=d\nn/d\rr$; then set $\RR=\rr$ and $\DD=\nn$. (We call this the \emph{regular choice}.)
\item Assuming the differential $d\nn$ is nonzero, let $\NN=d\rr/d\nn$; then set $\RR=\nn$ and $\DD=\rr$. (We call this the \emph{critical choice}.)
\end{itemize}
We remark that in either case we have
\begin{gather} \label{dddr}
\NN=\frac{{d}\DD}{{d}\RR}.
\end{gather}
To begin the process we always make a regular choice, but there are two possibilities. On $U$ the active coordinates are $x_1$ and $x_2$, either of which may be designated as the retained coordinate, and there is no deactivated coordinate.

Our labeling of charts will use the alphabet $\{1,2\}$. We will find it convenient to use the following notational device: $\overline{p}$ denotes the symbol opposite to $p$, i.e., $\overline{1}=2$ and $\overline{2}=1$.

Each chart on $S(k)$ will be labeled by $\cC(p_1p_2\cdots p_k)$, where $p_1p_2\cdots p_k$ is a string from the alphabet. In particular $\cC(\varnothing)$ is $U$ itself. In $\cC(p_1p_2\cdots p_k)$, we will use $2(k+1)$ coordinate names
\begin{gather*}
x_1, \enskip x_1(p_1), \enskip x_1(p_1 p_2), \enskip x_1(p_1 p_2 p_3), \enskip \ldots, \enskip x_1(p_1p_2\cdots p_k), \\
x_2, \enskip x_2(p_1), \enskip x_2(p_1 p_2), \enskip x_2(p_1 p_2 p_3), \enskip \ldots, \enskip x_2(p_1p_2\cdots p_k),
\end{gather*}
with $k$ of these names being redundant. The meaning of these coordinates is explained by recursion. In the chart $\cC(p_1p_2\cdots p_{k-1})$ on $S(k-1)$, the two active coordinates are $x_1(p_1p_2\cdots p_{k-1})$ and $x_2(p_1p_2\cdots p_{k-1})$. To build the chart $\cC(p_1p_2\cdots p_{k-1}1)$ on $S(k)$, we assume that the diffe\-ren\-tial $dx_1(p_1p_2\cdots p_{k-1})$ is nonzero. We introduce the new active coordinate
\begin{gather*}
x_{2}(p_1p_2\cdots p_{k-1}1):=\frac{{d}x_{2}(p_1p_2\cdots p_{k-1})}{{d}x_{1}(p_1p_2\cdots p_{k-1})}
\end{gather*}
and give a new (redundant) name to the retained active coordinate:
\begin{gather*}
x_{1}(p_1p_2\cdots p_{k-1}1):=x_{1}(p_1p_2\cdots p_{k-1}).
\end{gather*}
To build the chart $\cC(p_1p_2\cdots p_{k-1}2)$, we reverse the roles:
\begin{gather*}
x_{1}(p_1p_2\cdots p_{k-1}2):=\frac{{d}x_{1}(p_1p_2\cdots p_{k-1})}{{d}x_{2}(p_1p_2\cdots p_{k-1})},\\
x_{2}(p_1p_2\cdots p_{k-1}2):=x_{2}(p_1p_2\cdots p_{k-1}).
\end{gather*}
We remark that the regular choice leads to a coordinate chart $\cC(p_1p_2\cdots p_k)$ in which $p_{k-1}=p_k$, whereas the critical choice leads to a chart in which $p_{k-1} \neq p_k$.

For each chart $\cC(p_1p_2\cdots p_k)$, let $N(p_1p_2\cdots p_k)$ be the origin. In particular, the point $x_1=x_2=0$ is the origin of $U=\cC(\varnothing)$. The map $\pi_k\colon S(k)\to S(k-1)$ sends $N(p_1p_2\cdots p_k)$ to $N(p_1p_2\cdots p_{k-1})$.

\begin{Example} \label{beginrun}
This is the beginning of an example that will continue throughout the paper. In chart $\cC(212)$, the coordinates are
\begin{gather*}
x_1, \ x_2, \ x_1(2)=\frac{{d}x_1}{{d}x_2}, \ x_2(21)=\frac{{d}x_2}{{d}x_1(2)}, \ \text{and} \ x_1(212)=\frac{{d}x_1(2)}{{d}x_2(21)}.
\end{gather*}
The new active coordinate is $\nn=x_1(212)$; the retained active coordinate is $\rr=x_2(21)$; the other coordinates are inactive, including the deactivated coordinate $\dd=x_1(2)$.
\end{Example}

\begin{Lemma}$N(p_1p_2\cdots p_k)$ is the unique point of $S(k)$ lying over the origin of $U$ that is in $\cC(p_1p_2\cdots p_k)$ and in no other chart.
\end{Lemma}
\begin{proof} We use induction on $k$. Since $N(p_1p_2\cdots p_k)$ maps to $N(p_1p_2\cdots p_{k-1})$, it can only appear in chart $\cC(p_1p_2\cdots p_{k-1}1)$ or $\cC(p_1p_2\cdots p_{k-1}2)$. The coordinates of a point in these two charts are the same except for the last ones, which are reciprocals. Thus $N(p_1p_2\cdots p_k)$ only appears in $\cC(p_1p_2\cdots p_k)$.

Conversely, suppose $P \neq N(p_1p_2\cdots p_k)$ is a point of $\cC(p_1p_2\cdots p_k)$. Suppose $P = (0, 0, X_1,$ $X_2, \ldots, X_j, 0, \ldots, 0)$
in the coordinate system on $\cC(p_1p_2\cdots p_k)$, where $X_j \neq 0$. Then $P$ lies over the point $Q=(0,0,X_1,X_2,\ldots,X_j) \in \cC(p_1p_2\cdots p_j)$. The point $Q$ also appears in chart $\cC(p_1p_2\cdots \overline{p_j})$ as $(0,0,X_1,X_2,\ldots,1/X_j)$. Therefore $P$ also lies in some chart whose string begins with $p_1p_2\cdots \overline{p_j}$.
\end{proof}

\section{Prolongation of the nodal curve: scheme structure} \label{schemestruc}
We now begin our analysis of the prolongation of the nodal curve. Recall that we are studying the one-parameter family $\cX$ of curves $x_1 x_2=t$, and that our goal is to describe the total space~$\cX(k)$ of the family of $k$th order prolongations, whose central member $X(k)_0$ is called the \emph{$k$th prolongation of the nodal curve}. In this section, we prove the main results. In Theorem~\ref{yuugetheorem}, we describe $X(k)_0$; in Theorem~\ref{nodetwigmap}, we relate it to $X(k-1)_0$. We will prove the two theorems simultaneously. At the end of the section we will continue our running example.

For each chart $\cC=\cC(p_1p_2\cdots p_k)$ of $S(k)$, we denote the ideal of $\cX(k)$ in the coordinate ring of $\cC \times T$ by $\cI(k)$;
the ideal of $X(k)_0$ in the chart $\cC$ is denoted by $\cI(k)_0$. In our description of $X(k)_0$, we use a binomial $B(p_1p_2\cdots p_k)$.
These \emph{node binomials} are defined by recursion, beginning with the monomial $B(\varnothing)=x_1 x_2$. To describe this recursion, we use the
notations $\nn$, $\rr$, $\dd$ and $\NN$, $\RR$, $\DD$ as explained in Section~\ref{charts}. The binomial $B(p_1p_2\cdots p_{k-1})$ will be of the form
\begin{gather*}
\alpha \nn \rr + \beta \dd,
\end{gather*}
where $\alpha$ and $\beta$ are certain integers, and from it we derive two new binomials. (For $B(\varnothing)$ we have $\alpha=1$ and $\beta=0$.) If we make the regular choice, then the associated node binomial is
\begin{gather*}
B(p_1p_2\cdots p_{k}) = \alpha \NN\rr + (\alpha+\beta) \nn = \alpha \NN\RR + (\alpha+\beta) \DD
\end{gather*}
(where the new active variable $\NN$ represents $d\nn/d\rr$); if we make the critical choice, then the associated node binomial is
\begin{gather*}
B(p_1p_2\cdots p_{k}) = (\alpha + \beta) \NN\nn + \alpha \rr = (\alpha+\beta)\NN\RR + \alpha \DD
\end{gather*}
(with $\NN$ representing $d\rr/d\nn$). The first few node binomials are shown in Fig.~\ref{binomstolength3}.

The node binomials naturally arise from prolongation calculations in the relevant charts of~$S(k)$. We obtain them by recursion, using implicit differentiation and working in the appropriate charts. In the recursive step, we implicitly differentiate the equation
\begin{gather} \label{binomequalzero}
 \alpha\nn\rr + \beta\dd = 0 .
\end{gather}
By (\ref{dddr}), we know that $\nn=d\dd/d\rr$. If we make the regular choice, then we differentiate~(\ref{binomequalzero}) with respect to~$\rr$. This gives the equation
\begin{gather*} \alpha\frac{{d}\nn}{{d}\rr}\rr + \alpha\nn + \beta\frac{{d}\dd}{{d}\rr} = 0, \end{gather*}
which (since $\nn = \DD$) is
\begin{gather*} \alpha\NN\RR + \alpha\DD + \beta\DD = 0. \end{gather*}
If we make the critical choice, then we differentiate~(\ref{binomequalzero}) with respect to $\nn$, which gives the equation
\begin{gather*} \alpha\rr + \alpha\nn\frac{{d}\rr}{{d}\nn} + \beta\frac{{d}\dd}{{d}\rr}\frac{{d}\rr}{{d}\nn} = 0, \end{gather*}
which (since $\nn = \RR$) is
\begin{gather*} \alpha\DD+ \alpha\RR\NN + \beta\RR\NN = 0. \end{gather*}

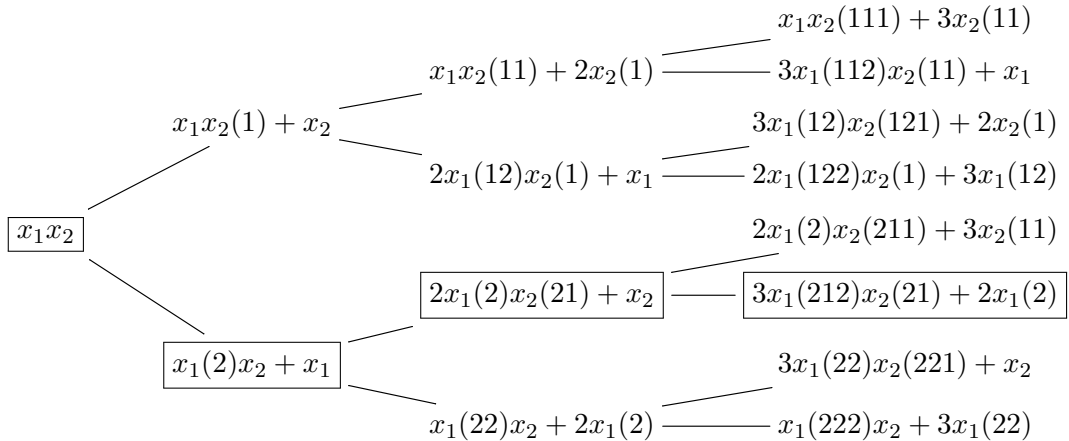
\begin{figure}[htbp]\centering
$
\xymatrix@R=0.25em{
& & & \ar@{-}[dl] x_1 x_2(111)+3x_2(11) \\
& & x_1 x_2(11) + 2x_2(1) \ar@{-}[dl] & 3x_1(112)x_2(11)+x_1 \ar@{-}[l] \\
& x_1 x_2(1) + x_2 \ar@{-}[ddl] & &3x_1(12)x_2(121)+2x_2(1) \ar@{-}[dl] \\
& & 2x_1(12)x_2(1)+x_1 \ar@{-}[ul] & 2x_1(122)x_2(1)+3x_1(12) \ar@{-}[l]\\
\boxed{x_1 x_2} & & & 2x_1(2)x_2(211)+3x_2(11) \ar@{-}[dl] \\
& & \boxed{2x_1(2)x_2(21)+x_2} \ar@{-}[dl] & \boxed{3x_1(212)x_2(21)+2x_1(2)} \ar@{-}[l] \\
& \boxed{x_1(2) x_2 + x_1} \ar@{-}[uul] & & 3x_1(22)x_2(221)+x_2 \ar@{-}[dl] \\
& & x_1(22) x_2 + 2x_1(2) \ar@{-}[ul] & x_1(222)x_2+3x_1(22) \ar@{-}[l] \\
 }
$
\caption{Node binomials up to length 3. The monomial $B(\varnothing)$ and the binomials $B(2)$, $B(21)$, $B(212)$ (here enclosed in boxes) are used in Example~\ref{runexample2}.} \label{binomstolength3}
\end{figure}

\begin{Theorem}\label{yuugetheorem}
The $k$th prolongation $X(k)_0$ of the nodal curve is a chain of $2^k +1$ irreducible curves (twigs) meeting at nodes.
\begin{enumerate}\itemsep=0pt
\item[$1.$] The nodes are the points $N(p_1p_2\cdots p_k)$, arranged in lexicographic order.
\item[$2.$] The end twigs are the $k$th prolongations of $x_2=0$ and of $x_1=0$.
\item[$3.$] The interior twigs $($taken with reduced structure$)$ are prolongations of vertical curves, and thus are copies of~$\PP^1$. In each chart, there are two twigs forming a pair of coordinate axes.
\item[$4.$] In chart $\cC(p_1p_2\cdots p_k)$, the ideal $\cI(k)_0$ is generated by
\begin{gather*}
B(\varnothing), B(p_1), B(p_1p_2), B(p_1p_2p_3), \ldots, B(p_1p_2\cdots p_k).
\end{gather*}
\end{enumerate}
\end{Theorem}
\begin{Theorem} \label{nodetwigmap} When the map $\pi_k\colon S(k)\to S(k-1)$ is applied to $X(k)_0$, one of the two twigs meeting at $N(p_1p_2\cdots p_k)$ is collapsed to the point $N(p_1p_2\cdots p_{k-1})$; the other is mapped isomorphically to a twig at level $k-1$.
\end{Theorem}
We call the twig that collapses to $N(p_1p_2\cdots p_{k-1})$ the \emph{emergent twig}, and we say that it \emph{emerges at level $k$}. The other twig is called the \emph{retained twig}. As we will see in the proof, the retained coordinate $\rr$ is the affine coordinate for the retained twig, and the new coordinate $\nn$ is the affine coordinate for the emergent twig.
\begin{proof}[Proof of Theorems \ref{yuugetheorem} and \ref{nodetwigmap}]
The prolongation calculations described just before the statement of Theorem \ref{yuugetheorem} produce binomials that must belong to the ideal $\cI(k)$; the process begins with $B(\varnothing)-t$. Indeed, at each nonsingular point of each member of the family, these are the calculations
for determining the higher-order curvilinear data. In chart $\cC(p_1p_2\cdots p_k)$, these calculations produce the sequence
\begin{gather*}
B(\varnothing)-t, B(p_1), B(p_1p_2), B(p_1p_2p_3), \ldots, B(p_1p_2\cdots p_k).
\end{gather*}
We will analyze the ideal~$\cB$ generated by this sequence and its corresponding family of schemes. We will prove that the central member
of this family satisfies the properties listed in parts (1) through (3) in the statement of Theorem~\ref{yuugetheorem}. We know that $\cB\subset\cI(k)$, but ultimately we will show that $\cB=\cI(k)$.

To begin, we claim that, in our selected chart $\cC(p_1p_2\cdots p_k)$, the central member of the family defined by $\cB$ has two irreducible
components
\begin{gather} \label{components}
 \begin{cases}
\nn = 0, \\
\text{each inactive coordinate} = 0,
\end{cases}
\qquad
 \begin{cases}
\rr = 0, \\
\text{each inactive coordinate} = 0.\end{cases}
\end{gather}
To see this, we use induction on $k$. By the inductive hypothesis, the ideal in the coordinate ring of $\cC(p_1p_2\cdots p_{k-1})$ generated by
\begin{gather} \label{oneleveldown}
B(\varnothing), \ldots, B(p_1p_2\cdots p_{k-1})
\end{gather}
has the two components of (\ref{components}). Now, looking in the coordinate ring of $\cC(p_1p_2\cdots p_k)$, we examine the ideal generated by~(\ref{oneleveldown}) and by
\begin{gather} \label{newgen}
B(p_1p_2\cdots p_k)=\alpha \NN \RR + \beta \DD.
\end{gather}
For any point satisfying the equations of (\ref{components}), all the inactive coordinates except $\DD$ vanish. Furthermore, one of the previous active coordinates is zero; thus, either $\DD=0$ or $\RR=0$. In the latter case, the vanishing of~(\ref{newgen}) tells us that we still have $\DD=0$; thus all inactive coordinates vanish. Again using~(\ref{newgen}), we conclude that either $\NN = 0$ or $\RR = 0$.

The second component in display (\ref{components}) is a copy of the affine line $\AAA^1$, with affine coordinate~$\nn$. The map $\pi_k \colon S(k) \to S(k-1)$ forgets this coordinate; thus it collapses this component to the point $N(p_1p_2\cdots p_{k-1})$. Letting $\nn$ go to infinity takes us outside the chart, and we reach the point $N(p_1p_2\cdots p_{k-1}\overline{p_k}) \in \cC(p_1p_2\cdots p_{k-1}\overline{p_k})$. Thus there is a copy of $\PP^1$, carrying the lexicographically adjacent nodes $N(p_1p_2\cdots p_{k-1}p_k)$ and $N(p_1p_2\cdots p_{k-1}\overline{p_k})$, and $\pi_k$ collapses it to the single point $N(p_1p_2\cdots p_{k-1})$. This is the emergent twig described in the paragraph after Theorem~\ref{nodetwigmap}.

The first component in display (\ref{components}) is a copy of the affine line $\AAA^1$, with affine coordinate $\rr$. The nature of this component depends upon the nature of the string $p_1 p_2 \cdots p_k$. There are two possibilities:
\begin{enumerate}\itemsep=0pt
\item[1)] there is a $j < k$ for which the string is $p_1\cdots p_{j-1}p_j \overline{p_j}\,\overline{p_j}\cdots\overline{p_j}$;
\item[2)] the string consists entirely of $1$'s or entirely of $2$'s.
\end{enumerate}

In case (1), the coordinate
\begin{gather*}
\rr = x_{\overline{p_j}}(p_1 \cdots p_j)
\end{gather*}
is the new coordinate in chart $\cC(p_1p_2\cdots p_j)$, and at level $j$ there is an emergent twig defined by the vanishing of all other coordinates. The $k-j$ coordinates introduced in going from level $j$ to level $k$ are the coordinates
\begin{gather*}
x_{p_j}(p_1 \cdots p_j \overline{p_j}) = \frac{dx_{p_j}(p_1 \cdots p_j)}{dx_{\overline{p_j}}(p_1 \cdots p_j)}, \\
x_{p_j}(p_1 \cdots p_j \overline{p_j} \, \overline{p_j}) = \frac{dx_{p_j}(p_1 \cdots p_j\overline{p_j})}{dx_{\overline{p_j}}(p_1 \cdots p_j)}, \\
x_{p_j}(p_1 \cdots p_j \overline{p_j} \, \overline{p_j} \, \overline{p_j})= \frac{dx_{p_j}(p_1 \cdots p_j \overline{p_j} \, \overline{p_j})}{dx_{\overline{p_j}}(p_1 \cdots p_j)}, \qquad \text{etc.}
\end{gather*}
obtained by differentiation with respect to $\rr$. Since $x_{p_j}(p_1 \cdots p_j)$ vanishes on this twig, each of these coordinates vanishes on its prolongation. Thus the first component of~(\ref{components}) is part of the $(k-j)$th prolongation of this twig. This emergent twig also appears
in chart $\cC(p_1p_2\cdots p_{j-1} \overline{p_j})$, in which the new coordinate is the reciprocal
\begin{gather*}
1/\rr = x_{p_j}(p_1 \cdots p_{j-1} \overline{p_j}).
\end{gather*}
In chart $\cC(p_1p_2\cdots p_{j-1} \overline{p_j} p_j \cdots p_j)$ of level $k$, the newly-introduced coordinates are
\begin{gather*}
x_{\overline{p_j}}(p_1 \cdots p_{j-1} \overline{p_j} p_j) = \frac{dx_{\overline{p_j}}(p_1 \cdots p_{j-1}\overline{p_j})}{dx_{p_j}(p_1 \cdots p_{j-1}\overline{p_j})}, \\
x_{\overline{p_j}}(p_1 \cdots p_{j-1} \overline{p_j} p_j p_j) = \frac{dx_{\overline{p_j}}(p_1 \cdots p_{j-1}\overline{p_j}p_j)}{dx_{p_j}(p_1 \cdots p_{j-1}\overline{p_j})}, \\
x_{\overline{p_j}}(p_1 \cdots p_{j-1} \overline{p_j} p_j p_j p_j)= \frac{dx_{\overline{p_j}}(p_1 \cdots p_{j-1}\overline{p_j}p_j p_j)}{dx_{p_j}(p_1 \cdots p_{j-1}\overline{p_j})}, \qquad \text{etc.}
\end{gather*}
Again, since the coordinate function $x_{\overline{p_j}}(p_1 \cdots p_{j-1}\overline{p_j})$ vanishes on the twig, each of these coordinates vanishes on its prolongation. Thus we see that, if we let $\rr$ go to infinity on the first component of~(\ref{components}), we reach the point $N(p_1p_2\cdots p_{j-1} \overline{p_j} p_j \cdots p_j)$; this is the lexicographically adjacent node.

We now consider case (2). If the string $p_1 p_2 \cdots p_k$ consists entirely of 1's, then the first component of (\ref{components}) is the $k$th prolongation of $x_2=0$; if it consists entirely of 2's, then the first component is the $k$th prolongation of $x_1=0$.

We now observe that the central member is a complete intersection: its coordinate ring is generated by $k+2$ indeterminates, subject to $k+1$ relations, and we have just seen that each component is one-dimensional. Thus there are no embedded components, and one can deduce the scheme structure of the central member simply by localizing and completing at the origin. Carrying out this localization, we now claim that the coordinate ring of the family (when localized and completed) is isomorphic to
\begin{gather*}
\frac{\kk[[\nn, \rr,t]]}{\langle\nn^{\alpha}\rr^{\alpha+\beta}-t\rangle}.
\end{gather*}
Again the argument is by induction on $k$. Suppose that we have established this isomorphism at level $k-1$. If we make the regular choice, then at level $k$ our ring is isomorphic to
\begin{gather*}
\frac{\kk[[\DD, \RR,t,\NN]]} {\langle\DD^{\alpha}\RR^{\alpha+\beta}-t, \alpha \NN\RR + (\alpha+\beta) \DD\rangle}
\end{gather*}
and we can eliminate $\DD$ from the presentation to obtain
\begin{gather*}
\frac{\kk[[\NN, \RR,t]]} {\langle\NN^{\alpha}\RR^{\alpha+(\alpha+\beta)}-t\rangle}
\end{gather*}
as required. If we make the critical choice, then at level $k$ our ring is isomorphic to
\begin{gather*}
\frac{\kk[[\RR, \DD,t,\NN]]} {\langle\RR^{\alpha}\DD^{\alpha+\beta}-t,(\alpha+\beta)\NN\RR + \alpha \DD\rangle}
\end{gather*}
and elimination of $\DD$ leads to
\begin{gather*}
\frac{\kk[[\NN, \RR,t]]} {\langle\NN^{\alpha+\beta}\RR^{\alpha+(\alpha+\beta)}-t\rangle} .
\end{gather*}

In this presentation of our (localized and completed) ring, it is clear that $t$ is not a zero divisor. Thus the family is flat. By uniqueness of the flat extension \cite[Proposition 9.8]{MR0463157}, we have $\cB=\cI(k)$. Specializing to $t=0$, we obtain the statement of part (4) of the theorem.
\end{proof}

Note that the image of an interior twig under $\pi_k$ is determined by the images of its two nodes: if the nodes map to the same node, then the entire twig is mapped to this node as well; otherwise it maps to the twig determined by the two image nodes. Furthermore, every interior twig
emerges at some level: one twig emerges from the initial node at level~1, two twigs then emerge at level 2, four at level 3, etc. Theorem~\ref{nodetwigmap} provides a natural way to label an interior twig on the $k$th prolongation of the nodal curve, as follows.
If the twig emerges at this level, then give it the label of the node from which it emerges; in other words, the twig emerging from $N(p_1\cdots p_{k-1})$ is $T(p_1\cdots p_{k-1})$. If the twig maps isomorphically to a twig at level $k-1$, then continue to use the label of the image twig; of course this means that the length of the label reflects the level at which the twig emerges. Concretely, the label of a twig is obtained from the
labels of its nodes by using the longest possible common initial string. Refer to Fig.~\ref{thirdprolongation} to see how this works at level~3.

\begin{Example} \label{runexample2}
The boxes in Fig.~\ref{binomstolength3} show the node binomials that define the third prolongation of the nodal curve in chart~$\cC(212)$. As described in (\ref{components}), the two twigs meeting at $N(212)$ are the retained twig $T(2)$:
\begin{gather*}\begin{cases}
x_1(212) = 0, \\
x_1 = x_2 = x_1(2) = 0,\end{cases}
\end{gather*}
for which the affine coordinate is $\rr=x_2(21)$, and the emergent twig $T(21)$:
\begin{gather*}\begin{cases}
x_2(21) = 0, \\
x_1 = x_2 = x_1(2) = 0,\end{cases}
\end{gather*}
for which the affine coordinate is $\nn=x_1(212)$. The point at infinity on the retained twig (i.e., the point at which $\rr$ becomes infinite) is the node $N(221)$; here $x_2(221)=0$. The point at infinity on the emergent twig is the node $N(211)$; here $x_2(211)=0$.
\end{Example}

\section{Cycle structure} \label{cyclestruc}
Here we work out the cycle class of $X(k)_0$, in other words, we deduce the multiplicities of the twigs in the chain of curves.
\begin{Theorem}\label{multiplicity}
In chart $\cC(p_1p_2\cdots p_k)$, the node binomial $B(p_1p_2\cdots p_k)=\alpha \nn\rr + \beta \dd$ determines the multiplicities of the two twigs:
the multiplicity of the retained twig is $\alpha$, and the multiplicity of the emergent twig is $\alpha+\beta$.
\end{Theorem}
\begin{proof} By our arguments in the proof of Theorem \ref{yuugetheorem}, we know that at each node the (localized and completed) coordinate ring for $X(k)_0$ is isomorphic to
\begin{gather*}
\frac{\kk[[\nn, \rr,t]]} {\langle\nn^{\alpha}\rr^{\alpha+\beta}\rangle}.\tag*{\qed}
\end{gather*}\renewcommand{\qed}{}
\end{proof}

\begin{Example} Continuing the running example, we refer to the rightmost box of Fig.~\ref{binomstolength3}. The retained twig $x_1=x_2=x_1(2)=x_1(212)=0$ has multiplicity~3. The emergent twig $x_1=x_2=x_1(2)=x_2(21)=0$ has multiplicity~5.
\end{Example}

One can also generate these multiplicities without invoking the node binomials. The \emph{twig multiplicity sequence} gives the multiplicities from left to right. It is generated recursively, beginning with $m_0$, the sequence~1,~1. The sequence $m_{k}$ is obtained from $m_{k-1}$ by inserting between each pair of consecutive terms a new term whose value is their sum, so that it has $2^k+1$ terms altogether. Thus we have

\smallskip

 \begin{tabular}{cc}
$m_0$: & 1 , 1 \\
$m_1$: & 1 , 2 , 1 \\
$m_2$: & 1 , 3 , 2 , 3 , 1 \\
$m_3$: & 1 , 4 , 3 , 5 , 2 , 5 , 3 , 4 , 1
 \end{tabular}

\smallskip

\noindent
and so on.

\section{Stratification} \label{strat}
Together with Castro and Shanbrom, we have worked out in \cite{MR3720327} the details of a natural coarse stratification of the monster spaces. We briefly recall that theory here, specializing to the case where the base is a surface. The basic idea is to apply the construction of Section~\ref{monster} in a different way. Assuming that $k \geq 2$, we begin with $M=S(k-1)$ and $\cB$ the bundle of vertical vectors at level $k-1$. This is a rank 1 bundle, and thus $\widetilde{M}$ is a copy of $S(k-1)$ inside $S(k)$; moreover it is a divisor meeting each $\PP^1$-fiber in one point. We call $\widetilde{M}$ the \emph{divisor at infinity} and denote it by $I_k$. Iterating the process, we obtain subspaces $I_k[1]$ in $S(k+1)$, then $I_k[2]$ in $S(k+2)$, etc. (By convention $I_k[0]=I_k$.) Each $I_k[n]$ is a copy of $S(k-1)$, having codimension $n+1$ inside $S(k+n)$. We remark that if we prolong the fiber of $S(k-1)$ over a point of $S(k-2)$, i.e., a~vertical curve then its successive prolongations lie on these $I_k[n]$'s.

In \cite{MR3720327}, we examined the \emph{intersection locus}
\begin{gather} \label{intloceq}
I_W = \bigcap_{j=2}^k I_j[n_j-1],
\end{gather}
a subspace of $S(k)$. We will explain the notation on the left side of (\ref{intloceq}) in a moment. To make sense of the notations on the right side, we need to make several remarks. First, each~$n_j$ is a nonnegative integer with $n_j \leq k-j+1$. The space $I_j[n_j-1]$ has been defined to be a~subspace of $S(n_j+j-1)$; in (\ref{intloceq}) we use its complete inverse image in $S(k)$. We have already remarked that $I_j[0]$ means $I_j$, and we interpret $I_j[-1]$ as $S(k)$ itself (so that this intersectand could be omitted). With these conventions, the intersection locus of~(\ref{intloceq}), if it is not empty, is nonsingular and of codimension $n_2+n_3+\cdots +n_k$.

To understand when these loci are nonempty, we work with certain \emph{code words}
\begin{gather} \label{genericword}
W=V_{A_1}V_{A_2}V_{A_3}\cdots V_{A_k}
\end{gather}
created according to the following rules:
\begin{enumerate}\itemsep=0pt
\item[1)] the first subscript $A_1$ must be $\varnothing$;
\item[2)] for $j \geq 2$, the subscript $A_j$ must be $\varnothing$ or $j$ or the prior subscript $A_{j-1}$.
\end{enumerate}
The symbol $V_{\varnothing}$ is also denoted $R$; in particular the first symbol is always $R$. The number of code words of length $k$ is the Fibonacci number $F_{2k-1}$ (where $F_1=F_2=1$). Let $n_j$ denote the total number of times that $j$ appears in the subscripts of $W$; then the intersection locus $I_W$ specified by~(\ref{intloceq}) is nonempty, and this method accounts for all nonempty intersection loci.
\begin{Example} \label{5strata} There are five code words of length three, and thus five associated intersection loci on $S(3)$:
\begin{gather*}
RRR\longleftrightarrow S(3) = I_2[-1] \cap I_3[-1], \\
RRV_3 \longleftrightarrow I_3 = I_2[-1] \cap I_3[0], \\
RV_2R \longleftrightarrow I_2 = I_2[0] \cap I_3[-1], \\
RV_2V_2 \longleftrightarrow I_2[1] = I_2[1] \cap I_3[-1], \\
RV_2V_3 \longleftrightarrow I_2 \cap I_3 = I_2[0] \cap I_3[0].
\end{gather*}
\end{Example}

The intersection loci of (\ref{intloceq}) are nested in a straightforward way: given code words~$W'$ and~$W$ whose respective associated integers satisfy $n'_j \geq n_j$ (for each $j=2,\ldots,k$), the intersection locus~$I_{W'}$ is contained in~$I_W$. Removing all the lower-dimensional loci from an intersection locus gives a \emph{code word stratum}, an open dense subset of the intersection locus. In this manner we have stratified~$S(k)$, as in~\cite{MR3720327}. In that paper we also gave equations defining each intersection locus in each chart; all of them are linear equations in the local coordinates. Each stratum is obtained by removing intersection loci of lower dimension; thus in each chart, each stratum is given by certain linear equations and linear inequalities in the local coordinates. For present purposes the detailed structure of these equations and inequalities is not needed.

Now recall our discussion in Section~\ref{monster} of the difference in viewpoints between the study of Goursat distributions and the study of curvilinear data. From the former perspective, one cannot distinguish, at level $2$, between regular and vertical data. Thus in the associated literature there is no possibility of using the symbol~$V_2$; indeed, often the code words in~(\ref{genericword}) are truncated by removing the first two symbols. To say this another way, the choice of vertical directions at level~$1$ is somewhat arbitrary, and consequently the divisor at infinity $I_2$ is not well-defined. Thus the stratification of a monster space is, with this viewpoint, slightly coarser than what we are using in this paper.

Our earliest studies of the monster (Semple) tower \cite{MR1100358,MR1143546, MR1287696} were focused on deriving formulas in enumerative geometry for algebraic curves in the projective plane; thus we were concerned with orbits of the monster spaces under the action of the projective linear group ${\rm PGL}(2)$. In this context, one wants an even finer stratification, in which, at level~$2$, there are three sorts of points:
\begin{enumerate}\itemsep=0pt
\item[1)] a general point of $\PP^2(2)$ represents the curvilinear data of a nondegenerate conic at a~point;
\item[2)] there is still a divisor at infinity, whose points represent the data of certain singular curves;
\item[3)] there is a second divisor whose points represents the data of lines, and this is a distinguished locus invariant under the action of the group.
\end{enumerate}
From this viewpoint, there are three possible symbols for the second entry of a code word. In our cited works, the symbols were $0$, $-$, and $\infty$; to pass to the code word of the present paper, replace $0$ or $-$ by $R$, and replace $\infty$ by $V_2$. In this finer stratification, there are eight strata at level~$3$, each of which happens to be a ${\rm PGL}(2)$-orbit, as enumerated in \cite[Theorem~2]{MR1143546}, rather than the five strata listed in Example~\ref{5strata} above or the two strata of the differential-geometric theory.

We now return to our analysis of the family acquiring a node.
\begin{Lemma}Every point of a twig belongs to the same stratum, except possibly its nodes.
\end{Lemma}
\begin{proof}
There is a chart in which the entire twig appears as a coordinate axis, except that one of its nodes is at infinity.
\end{proof}

Thus to each twig we can associate the code word of its stratum, called its \emph{twig word}. Similarly, the code word of the stratum containing a node will be called its \emph{node word}. The node word of $N(p_1p_2\cdots p_k)$ will be denoted $W(p_1p_2\cdots p_k)$.

\begin{Theorem} \label{twigwordrecursion}
Suppose that $0 \leq j \leq k-1$. The twig word of $T(p_1p_2\cdots p_j)$ $($a twig at level~$k)$ consists of $W(p_1p_2\cdots p_j)$ $($a node word at level~$j)$ followed by a single $R$ and a string of $k-1-j$ occurrences of $V_{j+2}$.
\end{Theorem}
Note that if $j=k-1$, then the final string is empty.

\begin{proof} The map from $S(k)$ to $S(j+1)$ is an isomorphism from $T(p_1p_2\cdots p_j)$ to the vertical curve over the node $N(p_1p_2\cdots p_j)\in S(j)$. (If $k=j+1$, the map is the identity.) The curvilinear data at level $j+1$ is thus regular: it represents a tangent direction at $N(p_1p_2\cdots p_j)$. Beyond this level we have the successive prolongations of this vertical curve and, as we have already observed, they lie within the spaces $I_{j+2}[n]$.
\end{proof}

\begin{Lemma} \label{nodewordrecursion} Consider the twig words of the two twigs meeting at $N(p_1p_2\cdots p_k)$. In each position, either the symbols agree, or one of the symbols is $R$ and the other is some $V_i$. The node word $W(p_1p_2\cdots p_k)$ can be deduced from these twig words as follows: in each position, if the symbols of the twig words agree, then use the common symbol; if they disagree, then use $V_i$ rather than $R$.
\end{Lemma}
\begin{proof}The divisors at infinity and their prolongations are closed loci. If the non-nodal points of a twig are contained in such a divisor or its prolongation, then so are the nodes of this twig. Thus if a symbol $V_i$ occurs in a certain position of a twig word, then the words of its nodes must have the same symbol in this position. For the same reason, the words of the two twigs meeting at a node cannot have differing $V$'s
in the same position. If both words have $R$ in some position, this means that the coordinates of the non-nodal points of both twigs satisfy a certain set of strict linear inequalities. Since each twig is a coordinate axis, the origin also satisfies these same inequalities.
\end{proof}

Using Theorem \ref{twigwordrecursion} and Lemma \ref{nodewordrecursion}, one can recursively compute any desired node word or twig word, invoking the following base cases: $W(\varnothing)=\varnothing$ and $W(11\cdots1)=W(22\cdots2)=RR\cdots R$.

\begin{Example} \label{nodewordcomputation} This continues and extends the running example. Fig.~\ref{nodewordexample} shows the computation of the node word $RV_2V_3V_3V_5$ for the node $N(21221)$. The upper tree shows the labels of the nodes to which $N(21221)$ maps, together with the twigs at these nodes; a slanted and dotted segment indicates that the twig emerges from the node shown to its left. The lower tree gives the corresponding node words and twig words, which can be deduced by working from left to right: first use Theorem~\ref{twigwordrecursion} to compute the twig words, and then use Lemma~\ref{nodewordrecursion} to compute the node word.
\end{Example}

\begin{figure}[htbp]
\begin{gather*}
\xymatrix@R=0.25em{
& T(\varnothing) \ar@{--}[r]
& T(\varnothing)
& T(21)
& T(212) \ar@{--}[r]
& T(212)
\\
\boxed{N(\varnothing)} \ar@{-}[r] \ar@{--}[ru]
& \boxed{N(2)} \ar@{-}[r] \ar@{--}[rd]
& \boxed{N(21)} \ar@{-}[r] \ar@{--}[ru]
& \boxed{N(212)} \ar@{-}[r] \ar@{--}[ru]
& \boxed{N(2122)} \ar@{-}[r] \ar@{--}[rd]
& \boxed{N(21221)}
\\
& \text{right end}
& T(2) \ar@{--}[r]
& T(2) \ar@{--}[r]
& T(2)
& T(2122)
\\
\\
\\
& R \ar@{--}[r]
& RV_2
& RV_2R
& RV_2V_3R \ar@{--}[r]
& RV_2V_3RV_5
\\
\boxed{\varnothing} \ar@{-}[r] \ar@{--}[ru]
& \boxed{R} \ar@{-}[r] \ar@{--}[rd]
& \boxed{RV_2} \ar@{-}[r] \ar@{--}[ru]
& \boxed{RV_2V_3} \ar@{-}[r] \ar@{--}[ru]
& \boxed{RV_2V_3V_3} \ar@{-}[r] \ar@{--}[rd]
& \boxed{RV_2V_3V_3V_5}
\\
& R
& RR \ar@{--}[r]
& RRV_3\ar@{--}[r]
& RRV_3V_3
& RRV_3V_3R
}
\end{gather*}
\caption{Recursive computation of the node word of the node $N(21221)$. Refer to Example \ref{nodewordcomputation} for an explanation.}\label{nodewordexample}
\end{figure}
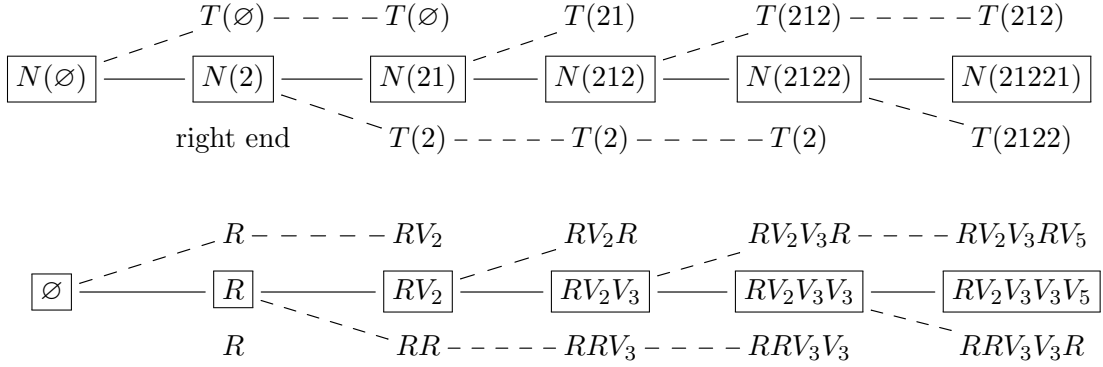

One can also use Theorem \ref{twigwordrecursion} and Lemma \ref{nodewordrecursion} to derive the following explicit formula.
\begin{Theorem} The node word $W(p_1p_2\cdots p_k)$ can be obtained from $p_1p_2\cdots p_k$ by the following process: replace its initial block of the symbol $1$ or $2$ by the same-length block of the symbol~$R$, and replace each succeeding block by the same-length block of the symbol~$V_j$, where $j$ is the beginning position of the block.
\end{Theorem}
For example, $W(222122112)$ is $RRRV_4V_5V_5V_7V_7V_9$.
\begin{proof} We prove this by induction on the length of the label. We have already remarked that $W(\varnothing)=\varnothing$ and $W(11\cdots1)=W(22\cdots2)=RR\cdots R$, since the nodes are end nodes.

Now consider a node with label $p_1p_2\cdots p_{j-2}1 2^{k-j+1}$, ending in a block of length \smash{$k-j+1 \geq 1$}. The node to its left has label $p_1p_2\cdots p_{j-2}1 2^{k-j}1$, and thus the twig to its left has label \linebreak $p_1p_2\cdots p_{j-2}1 2^{k-j}$. The node to its right has label $p_1p_2\cdots p_{j-2} 2 1^{k-j+1}$, and thus the twig to its right has label $p_1p_2\cdots p_{j-2}$. By Theorem~\ref{twigwordrecursion} the twig words are, respectively,
\begin{gather*}
W\big(p_1p_2\cdots p_{j-2}1 2^{k-j}\big)R \qquad \text{and} \qquad W(p_1p_2\cdots p_{j-2})RV_j^{k-j+1}.
\end{gather*}
By the inductive hypothesis, these two words agree through position $j-2$. We now apply Lemma~\ref{nodewordrecursion}. In the next position the second twig word has~$R$, so that our node word begins with $W(p_1p_2\cdots p_{j-2}1)$. After that, the first twig word ends with $V_j^{k-j}R$ and the second twig word ends with $V_j^{k-j+1}$; thus our node word ends with $V_j^{k-j+1}$.

For a node with label $p_1p_2\cdots p_{j-2}2 1^{k-j+1}$, the same arguments apply \emph{mutatis mutandis}.
\end{proof}

\begin{Example}Fig.~\ref{thirdprolongationwords} shows all node words and twig words of the third prolongation of the nodal curve.
\end{Example}

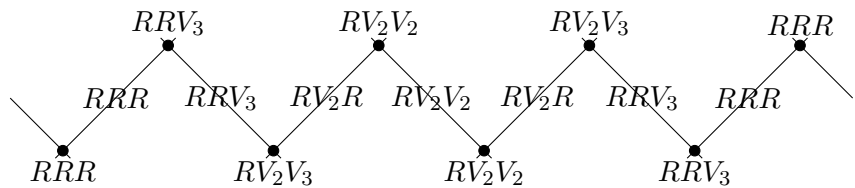
\begin{figure}[htbp]\centering
\begin{tikzpicture}
\draw (0.8,-0.8) --(1.6,-1.6);
\draw (2.8,0) --(4.4,-1.6);
\draw (5.6,0) --(7.2,-1.6);
\draw (8.4,0) --(10.0,-1.6);
\draw (11.2,0) --(12.0,-0.8);

\draw (1.4,-1.6) --(3.0,0);
\draw (4.2,-1.6) --(5.8,0);
\draw (7.0,-1.6) --(8.6,0);
\draw (9.8,-1.6) --(11.4,0);

\draw[fill] (1.5,-1.5) circle [radius=0.07];
\draw[fill] (4.3,-1.5) circle [radius=0.07];
\draw[fill] (7.1,-1.5) circle [radius=0.07];
\draw[fill] (9.9,-1.5) circle [radius=0.07];

\draw[fill] (2.9,-0.1) circle [radius=0.07];
\draw[fill] (5.7,-0.1) circle [radius=0.07];
\draw[fill] (8.5,-0.1) circle [radius=0.07];
\draw[fill] (11.3,-0.1) circle [radius=0.07];

\node [below] at (1.5,-1.5) {$RRR$};
\node [below] at (4.3,-1.5) {$RV_2V_3$};
\node [below] at (7.1,-1.5) {$RV_2V_2$};
\node [below] at (9.9,-1.5) {$RRV_3$};

\node [above] at (2.9,-0.1) {$RRV_3$};
\node [above] at (5.7,-0.1) {$RV_2V_2$};
\node [above] at (8.5,-0.1) {$RV_2V_3$};
\node [above] at (11.3,-0.1) {$RRR$};

\node at (2.2,-0.8) {$RRR$};
\node at (3.6,-0.8) {$RRV_3$};
\node at (5.0,-0.8) {$RV_2R$};
\node at (6.4,-0.8) {$RV_2V_2$};
\node at (7.8,-0.8) {$RV_2R$};
\node at (9.2,-0.8) {$RRV_3$};
\node at (10.6,-0.8) {$RRR$};

\end{tikzpicture}

\caption{The third prolongation of the nodal curve, as shown in Fig.~\ref{thirdprolongation}, together with corresponding node words and twig words.} \label{thirdprolongationwords}
\end{figure}

\section{Speculations about other singularities} \label{speculate}
What should one expect if one carries out a similar analysis for a different sort of singular curve~$C$ on a surface? If one prolongs the curve by itself, without reference to any family, one obtains its \emph{strict prolongation}; for $xy=t$ that is simply the pair of end twigs. If one creates a~one-dimensional family in which~$C$ is the central member and all other members are nonsingular (a~\emph{smoothing family}), one expects that the construction of this paper will yield three sorts of structures:
\begin{enumerate}\itemsep=0pt
\item[1)] a configuration of curves, consisting of the strict prolongation together with prolongations of certain vertical curves;
\item[2)] a one-dimensional cycle obtained by assigning an appropriate multiplicity to each curve in the configuration;
\item[3)] a scheme whose fundamental class is this cycle.
\end{enumerate}

The authors of \cite{MR2434453} consider a locally irreducible curve, e.g., the $A_n$ singularity defined by $y^2=x^{n+1}$, and define a configuration of curves of the type just described. Their analysis shows that their configuration matches the configuration one obtains by the standard process of embedded resolution by point blowups, with the prolongations of vertical curves naturally matching up with the exceptional divisors of the standard process. (See \cite[Chapter~1]{MR2289519} for an account of embedded resolution of curve singularities.) We think that the same configuration will emerge via prolongation of any smoothing family.

Embedded resolution produces even more: it associates a multiplicity to each exceptional divisor, and thus defines a cycle (in fact a divisor, since the entire configuration lies on a~surface). Here again we speculate that the prolongation of any smoothing family will yield the same multiplicities, i.e., that the same cycle will emerge, no matter which smoothing family is used.

As for scheme structure, one should note that the versal deformation space is typically of higher dimension; for example, the versal deformation of the $A_n$ singularity consists of all curves
\begin{gather*}
y^2=x^{n+1} + t_n x^{n} + t_{n-1} x^{n-1} + \cdots + t_1 x + t_0.
\end{gather*}
One can obtain many different one-dimensional smoothing families by taking various curves within the base of the versal deformation. Here we anticipate that one can obtain many dif\-fe\-rent limiting scheme structures on the central fiber, which reflect, in some fashion, the various directions of approach on the base. We hope to pursue these speculations further.

\subsection*{Acknowledgments} \label{ack}
We thank the anonymous referees for their suggestions that improved this paper. We are also grateful to Corey Shanbrom and Richard Montgomery for many useful conversations. This work was partially supported by a grant from the Simons Foundation (\#318310 to Gary Kennedy).

\pdfbookmark[1]{References}{ref}
\LastPageEnding

\end{document}